\documentclass[12pt]{amsart}

\usepackage{amssymb}
\usepackage{amscd}

\title[A generalization of Esteves--Homma's example]{A generalization of Esteves--Homma's example of tangentially degenerate curves} 
\author{Satoru Fukasawa}

\address{Department of Mathematical Sciences, Faculty of Science, Yamagata University, Kojirakawa-machi 1-4-12, Yamagata 990-8560, Japan} 
\email{s.fukasawa@sci.kj.yamagata-u.ac.jp} 

\thanks{The author was partially supported by JSPS KAKENHI Grant Number JP19K03438.}

\newtheorem{theorem}{Theorem}
\newtheorem{definition}{Definition}

\newtheorem{proposition}{Proposition}
 
\newtheorem{question}{Question}

\theoremstyle{definition}
 
\newtheorem{remark}{Remark}

\begin{document}
\begin{abstract}
This paper presents a method of a construction of tangentially degenerate curves with a birational Gauss map, focusing on the non-classicality of automorphisms. 
This method describes a generalization of  Esteves--Homma's example of this kind.  
In addition, this paper presents a smooth projective curve with a birational Gauss map such that a general tangent line contains three or more points of the curve, which answers a question raised by Kaji in the affirmative. 
\end{abstract}

\maketitle 

\section{Introduction}

An irreducible projective curve $C \subset \mathbb{P}^N$ with $N \ge 3$ not contained in any plane over an algebraically closed field $k$ of characteristic $p \ge 0$ is said to be {\it tangentially degenerate} if $(T_PC\setminus \{P\}) \cap C \ne \emptyset$ for a general point $P \in C$, where $T_PC \subset \mathbb{P}^N$ is the projective tangent line at $P$.  
This terminology is due to Kaji \cite{kaji1}. 
In characteristic zero, the existence of tangentially degenerate curves was asked by Terracini in 1932 (\cite[p.143]{terracini}, see also \cite{ciliberto}). 
For the case where $p=0$ and the morphism $X \rightarrow \mathbb{P}^N$ induced by the normalization $\varphi: X \rightarrow C$ is unramified, the nonexistence of such curves was proved by Kaji \cite{kaji1} in 1986. 
Generalizations of Kaji's theorem were obtained in this century (\cite{bolognesi-pirola, kaji3}).  

In positive characteristic, there exist many examples of tangentially degenerate curves. 
In most of such known cases, the Gauss map $\gamma: C \dashrightarrow \mathbb{G}(1, \mathbb{P}^N)$, which sends a smooth point $P \in C$ to the tangent line $T_PC \in \mathbb{G}(1, \mathbb{P}^N)$, is {\it not} separable. 
Very surprisingly, in 1994, Esteves and Homma \cite{esteves-homma} presented an embedding 
$$ \varphi: \mathbb{P}^1 \rightarrow \mathbb{P}^3; \ (1:t) \mapsto (1:t:t^2-t^p:t^3+2t^p-3t^{p+1}) $$
of $\mathbb{P}^1$ such that $\varphi(\mathbb{P}^1)$ is tangentially degenerate and the Gauss map of $\varphi(\mathbb{P}^1)$ is birational onto its image. 
This is only known example of a tangentially degenerate curve with a separable Gauss map.  
The existence of another example has been unknown for a long time.  
One reason is that the separability of the Gauss map is equivalent to the reflexivity with respect to projective dual (\cite{hefez-kleiman, hefez-voloch, kaji2, voloch}), and that several pathological phenomena in positive characteristic do not occur under the assumption that the reflexivity holds (\cite{hefez-kleiman}). 

This paper proves the following:  

\begin{theorem} \label{rational}
Let $p>2$, let $q$ be a power of $p$, and let $\mathbb{F}_{q^n} \subset k$ be a finite field of $q^n$ elements containing the set $\{\alpha \in k \ | \ \alpha^{q-2}=1\}$. 
We consider the morphism
$$ \varphi: \mathbb{P}^1 \rightarrow \mathbb{P}^4; \ (1:t) \mapsto (1:t:t^2-t^q:t^{q^n}-t^{q^{2n}}:t(t^{q^n}-t^{q^{2n}})). $$
Then the following hold. 
\begin{itemize}
\item[(a)] $\varphi$ is an embedding. 
\item[(b)] For any point $P \in \mathbb{P}^1 \setminus \{(0:1)\}$, the set $(T_{\varphi(P)}\varphi(\mathbb{P}^1) \setminus \{\varphi(P)\}) \cap \varphi(\mathbb{P}^1)$ consists of exactly $q-2$ points. 
\item[(c)] The Gauss map of $\varphi(\mathbb{P}^1)$ is birational onto its image. 
\end{itemize}
\end{theorem}

Theorem \ref{rational} answers the following question raised by Kaji \cite{kaji4} in the affirmative. 
 
\begin{question} \label{kaji's question} 
Does there exist a tangentially degenerate space curve $C \subset \mathbb{P}^N$ with a birational Gauss map such that a general tangent line $T_PC$ of $C$ contains two points of $C$ other than $P$?
\end{question}

In the proof of Theorem \ref{rational}, it is important to show that $q-2$ automorphisms of $\mathbb{P}^1$ are {\it non-classical}, in the sense of Levcovitz \cite[p.136]{levcovitz}. 
In this paper, this term is used only in an extremal case. 

\begin{definition}[Levcovitz \cite{levcovitz}]
Let $X$ be a smooth projective curve, and let $\varphi: X \rightarrow \mathbb{P}^N$ with $N \ge 2$ be a morphism, which is birational onto its image.  
An automorphism $\sigma \in {\rm Aut}(X) \setminus \{1\}$ is said to be non-classical with respect to $\varphi$ 
if 
$$ \varphi \circ \sigma (P) \in \left\langle \varphi(P), \ \frac{d \varphi}{d x}(P) \right\rangle $$
for a general point $P$ of $X$, where $x$ is a local parameter at some point of $X$.   
\end{definition}

\begin{remark}
If $\sigma \in {\rm Aut}(X) \setminus \{1\}$ is non-classical with respect to a birational morphism $\varphi: X \rightarrow \mathbb{P}^N$ with $N \ge 3$ onto its image, then $\varphi(X)$ is tangentially degenerate. 
\end{remark} 

This paper presents a method of a construction of tangentially degenerate curves with a birational Gauss map, focusing on the non-classicality of automorphisms. 
In particular, this paper proves the following:  

\begin{theorem} \label{main} 
Let $p>2$ and $N \ge 3$. 
If a smooth projective curve $X$ admits a local parameter $x \in k(X)$ at some point $P \in X$, a function $y \in k(X)$, and an automorphism $\sigma \in {\rm Aut}(X)$ such that  
\begin{itemize}
\item[(a)] $y \not \in \langle 1, x \rangle$, 
\item[(b)] $k(X)=k(x, y)$,  
\item[(c)] $\sigma^*x=x+\alpha$ for some $\alpha \in k \setminus \{0\}$, and 
\item[(d)] $\sigma^*y-y=\alpha\frac{dy}{dx}$,
\end{itemize}
then there exists a morphism $\varphi: X \rightarrow \mathbb{P}^N$, which is birational onto its image, such that 
\begin{itemize}
\item[(1)] $\varphi(X)$ is not contained in any hyperplane of $\mathbb{P}^N$, 
\item[(2)] the automorphism $\sigma$ is non-classical with respect to $\varphi$, and
\item[(3)] the Gauss map of $\varphi(X)$ is birational onto its image.  
\end{itemize}
\end{theorem}

In the case where the order of $\sigma$ is $p$, assumptions can be simplified. 

\begin{theorem} \label{main2} 
Let $p>2$ and $N \ge 3$. 
If a smooth projective curve $X$ admits a local parameter $x \in k(X)$ at some point $P \in X$ and an automorphism $\sigma \in {\rm Aut}(X)$ such that  
\begin{itemize}
\item[(a)] $\sigma^*x-x \in k \setminus \{0\}$, and  
\item[(b)] the order of $\sigma$ is $p$,
\end{itemize}
then the same assertion as in Theorem \ref{main} holds. 
\end{theorem}

\begin{remark}
If we take $k(X)=k(t)$, $\alpha=1$, $x=t$ and $y=t^2-t^p$ (and $g_3=t^3+2t^p-3t^{p+1}$ in the proof of Theorem \ref{main}), then we can recover Esteves--Homma's embedding. 
\end{remark}

\begin{remark}
Conditions (a) and (b) in Theorem \ref{main2} are satisfied for {\it many} examples of Artin--Schreier curves. 
A curve $X$ with a function field $k(x, y)$ given by an irreducible polynomial $x^q-x-g(y) \in k[x, y]$ such that $g(y) \in k[y]$ and $g_y \ne 0$ is such an example.  
\end{remark} 

\section{Proofs of main theorems} 

This paper introduces a vector subspace $V_{\sigma, x} \subset k(X)$ and investigates it, inspired by Esteves--Homma's functions $t^2-t^p$ and $t^3+2t^p-3t^{p+1}$. 
Let $x$ be a local parameter of a smooth projective curve $X$ at a point $P \in X$, let $\sigma \in {\rm Aut}(X) \setminus \{1\}$, and let $f:=\sigma^*x-x \in k(X)$. 
We define the set  
$$ V_{\sigma, x}:=\left\{g \in k(X) \ | \ \sigma^*g-g=f\frac{dg}{dx}\right\}. $$

\begin{proposition} \label{functions} 
Assume that $p>0$ and $f \ne 0$.  
Then the following hold. 
\begin{itemize}
\item[(a)] The set $V_{\sigma, x}$ is a vector subspace of $k(X)$ over $k$ with $x \in V_{\sigma, x}$. 
\item[(b)] The field $(k(X)^{\sigma})^p$ is a subspace of $V_{\sigma, x}$. 
Furthermore, if the order of $\sigma$ is finite, then the dimensions of $(k(X)^{\sigma})^p$ and $V_{\sigma, x}$ are infinite over $k$. 
\item[(c)] If the order of $\sigma$ is prime, then $k(X)=k(x, y)$ for some $y \in ((k(X)^\sigma)^p) \setminus k$. 
In this case, $y \not \in \langle 1, x \rangle$. 
\item[(d)] Let $n \ge 1$ be an integer, let $\alpha \in k \setminus \{0\}$, and let $\beta \in k$ with $\beta\alpha^{p^n}+\alpha^2=0$. 
If $f=\alpha$, that is, $\sigma^*x=x+\alpha$, then $x^2+\beta x^{p^n} \in V_{\sigma, x}$.  
\item[(e)] If $f \in k \setminus \{0\}$, then the order of $\sigma$ is finite, and it is divisible by $p$. 
\end{itemize} 
\end{proposition}

\begin{proof} 
Let $g, h \in V_{\sigma, x}$ and $a, b \in k$. 
Then 
\begin{eqnarray*}
\sigma^*(a g+b h)-(a g+b h) &=& a \sigma^* g+ b \sigma^* h-(a g+ b h)=a (\sigma^*g-g)+b (\sigma^* h-h) \\
&=& a f\frac{dg}{dx}+b f\frac{dh}{dx}=f\frac{d}{d x}(a g+b h). 
\end{eqnarray*} 
On the other hand, 
$$ \sigma^*x-x=f=f\frac{d x}{d x}. $$
Assertion (a) follows. 

For each $g \in (k(X)^\sigma)^p$, 
$$ \sigma^*g-g=0, \ \frac{dg}{dx}=0. $$
This implies that $(k(X)^\sigma)^p \subset V_{\sigma, x}$.  
Assume that the order of $\sigma$ is finite. 
Then $[k(X): k(X)^\sigma]$ is finite.  
Since $[k(X)^{\sigma}: (k(X)^{\sigma})^p]$ is finite, it follows that $[k(X):(k(X)^{\sigma})^p]$ is finite, and hence, the dimension of $(k(X)^\sigma)^p$ is infinite. 
Assertion (b) follows. 

Assume that the order of $\sigma$ is prime. 
Since $[k(X):k(X)^{\sigma}]$ is prime and $x \not\in k(X)^{\sigma}$, it follows that $k(X)=(k(X)^{\sigma})(x)$. 
It is inferred that $k(X)^p \subset (k(X)^\sigma)^p(x)$. 
Since $k(X)/k(x)$ is finite and separable, there exists $z \in k(X)^p$ such that $k(X)=k(x, z)$ (see, for example, \cite[Proposition 3.10.2]{stichtenoth}). 
Therefore, there exist $y_1, \ldots, y_n \in (k(X)^{\sigma})^p$ such that $k(X)=k(x)(y_1, \ldots, y_n)$. 
Since $k(X)/k(x)$ is separable, there exists $y \in (k(X)^{\sigma})^p$ such that $k(X)=k(x, y)$. 
If $y \in k$, then $k(X)=k(x)$, and we can take another $y \in (k(X)^\sigma)^p \setminus k$ with $k(x)=k(x, y)$. 
Since $(k(X)^\sigma)^p \cap \langle 1, x \rangle=k$, it follows that $y \not\in \langle 1, x \rangle$. 
Assertion (c) follows. 

Let $g=x^2+\beta x^{p^n}$. 
Then 
\begin{eqnarray*}
\sigma^*g-g &=& (x+\alpha)^2+\beta (x+\alpha)^{p^n}-(x^2+\beta x^{p^n}) \\
&=& 2\alpha x+(\alpha^2+\beta \alpha^{p^n})=2 \alpha x = f \frac{d g}{d x},  
\end{eqnarray*}
and hence, $g \in V_{\sigma, x}$. 
Assertion (d) follows. 

Let $f=\alpha \in k$. 
Then $(\sigma^p)^*x=x+p \alpha=x$, and hence, $x \in k(X)^{\sigma^p}$. 
Since $x \in k(X)$ is transcendental over $k$, it follows that $[k(X):k(X)^{\sigma^p}]$ is finite. 
This implies that $(\sigma^p)^m=1$ for some positive integer $m$. 
On the contrary, if $\sigma^m=1$, then $x=(\sigma^m)^*x=x+m \alpha$. 
This implies that $m$ is divisible by $p$. 
Assertion (e) follows. 
\end{proof}

\begin{remark}
For any automorphism $\sigma \in {\rm Aut}(X) \setminus \{1\}$ and any point $P \in X$, there exists a local parameter $x \in k(X)$ at $P$ such that $\sigma^*x-x \ne 0$. 
\end{remark}

We prove main theorems. 

\begin{proof}[Proof of Theorem \ref{rational}] 
Let $P_\infty=(0:1) \in \mathbb{P}^1$. 
By the expression, $\mathbb{P}^1 \setminus \{P_\infty\} \cong \varphi(\mathbb{P}^1 \setminus \{P_\infty\})$ as affine varieties.    
We consider a neighborhood of $P_{\infty}$. 
The morphism $\varphi$ is represented by 
\begin{eqnarray*}
\varphi(s:1) &=& (s^{q^{2n}+1}:s^{q^{2n}}: s^{q^{2n}-1}-s^{q^{2n}+1-q}:s^{q^{2n}+1-q^{n}}-s: s^{q^{2n}-q^{n}}-1) \\
&=& \left (\frac{s^{q^{2n}+1}}{s^{q^{2n}-q^{n}}-1}:\frac{s^{q^{2n}}}{s^{q^{2n}-q^{n}}-1}: \frac{s^{q^{2n}-1}-s^{q^{2n}+1-q}}{s^{q^{2n}-q^{n}}-1}:s: 1\right).  
\end{eqnarray*}
Therefore, $\varphi(P_\infty)=(0:0:0:0:1)$, and the order of the hyperplane defined by $X_3=0$ at $P_\infty$ is equal to $1$. 
This implies that $\varphi(P_\infty)$ is a smooth point. 
Assertion (a) follows. 

Let $\sigma_\alpha: \mathbb{P}^1 \rightarrow \mathbb{P}^1$ be an automorphism given by $t \mapsto t+\alpha$ with $\alpha^{q-2}=1$. 
Since $\alpha \in \mathbb{F}_{q^n}$, it follows that $\alpha^{q^n}-\alpha^{q^{2n}}=(\alpha-\alpha^{q^n})^{q^n}=0$. 
Then
\begin{eqnarray*}
\frac{d \varphi}{d t} &=& (0, \ 1, \ 2t, \ 0, \ t^{q^n}-t^{q^{2n}}), \\
\varphi \circ \sigma_\alpha &=& (1, \ t+\alpha, \ t^2-t^q+2\alpha t, \ t^{q^n}-t^{q^{2n}}, \ (t+\alpha)(t^{q^n}-t^{q^{2n}})). 
\end{eqnarray*} 
Let $P=(1:t) \in \mathbb{P}^1 \setminus \{P_\infty\}$.
It follows that 
$$ \varphi(\sigma_\alpha(P)) =\varphi (P)+\alpha \frac{d\varphi}{d t}(P) \in \left\langle \varphi(P), \ \frac{d\varphi}{d t}(P) \right\rangle$$
for any $\alpha \in k$ with $\alpha^{q-2}=1$.
Therefore, there exist $q-2$ points of $(T_{\varphi(P)}\varphi(\mathbb{P}^1)  \setminus \{\varphi(P)\}) \cap \varphi(\mathbb{P}^1)$.  
It is not difficult to check that $\varphi(P_{\infty}) \not\in T_{\varphi(P)}\varphi(\mathbb{P}^1)$. 
Assume that $\varphi(P') \in T_{\varphi(P)}\varphi(\mathbb{P}^1)$ for a point $P'=(1:u) \in \mathbb{P}^1 \setminus \{P_\infty\}$. 
Then, for some $\beta \in k$,  
$$ u= t+\beta \ \mbox{ and } \ u^2-u^q=t^2-t^q+2 \beta t. $$
These imply that $u=t+\beta$ and $\beta^2-\beta^q=0$.  
Assertion (b) follows.

The tangent line is spanned by the row vectors of the matrix  
$$ 
\left(\begin{array}{c} 
\varphi \\
\frac{d \varphi}{d t}
\end{array} 
\right) \sim 
\left(\begin{array}{ccccc}
1 & 0 & -t^2-t^q & t^{q^n}-t^{q^{2n}} & 0 \\
0 & 1 & 2t & 0 & t^{q^n}-t^{q^{2n}} 
\end{array} \right). 
$$
The function field $k(\gamma \circ \varphi(\mathbb{P}^1))$ of the image of the Gauss map $\gamma$ contains the function $\frac{d(t^2-t^q)}{dt}=2t$. 
Since $p>2$, it follows that $t \in k(\gamma \circ \varphi(\mathbb{P}^1))$. 
This implies that the Gauss map $\gamma$ is birational onto its image. 
\end{proof}

\begin{proof}[Proof of Theorem \ref{main}] 
Assume that condition (c) is satisfied for an automorphism $\sigma \in {\rm Aut}(X)$.  
Let $V_{\sigma, x}$ be the vector space as in Proposition \ref{functions} with $f=\alpha$, and let $g_3, \ldots, g_N \in V_{\sigma, x}$. 
We consider the rational map 
$$ \varphi: X \dashrightarrow \mathbb{P}^N; \ (1:x:y:g_3:\cdots:g_N). $$
By condition (b), $\varphi$ is birational onto its image. 
By conditions (a) and (d), $y \not \in \langle 1, x \rangle$ and $y \in V_{\sigma, x}$. 
According to Proposition \ref{functions} (b) and (e), we can take $g_3, \ldots, g_N \in V_{\sigma, x}$ such that $\dim \langle 1, x, y, g_3, \ldots, g_N \rangle=N+1$. 
This implies assertion (1). 
Then 
\begin{eqnarray*}
\frac{d \varphi}{d x}&=&\left(0, \ 1, \ \frac{d y}{d x}, \ \ldots, \frac{d g_i}{d x}, \ldots \right), \\
\varphi\circ\sigma&=&(1, \ x+\alpha, \ \sigma^*y, \ \ldots, \sigma^*g_i, \ldots).  
\end{eqnarray*}
Since $1, x, y, g_3, \ldots, g_N \in V_{\sigma, x}$, it follows that 
$$ \varphi(\sigma(P))=\varphi(P)+\alpha\frac{d \varphi}{d x}(P) \in \left\langle \varphi(P), \ \frac{d \varphi}{d x}(P) \right\rangle$$
for a general point $P$ of $X$. 
Assertion (2) follows. 
The tangent line is spanned by the row vectors of the matrix  
$$ 
\left(\begin{array}{ccccccc}
1 & 0 & y-x \frac{d y}{d x} & \cdots & g_i-x \frac{d g_i}{d x} & \cdots\\
0 & 1 & \frac{d y}{d x} & \cdots & \frac{d g_i}{d x} & \cdots
\end{array} \right). 
$$
If we take $g_i=x^2+ \beta x^{p^n}$ with $\beta\alpha^{p^n}+\alpha^2=0$ for some $i$, as in Proposition \ref{functions} (d), then the function field $k(\gamma \circ \varphi(X))$ of the image of the Gauss map $\gamma$ contains the function $d g_i/d x=2x$. 
Since $p>2$, it follows that $x \in k(\gamma \circ \varphi(X))$. 
In this case, since $dy/dx, y-x(dy/dx) \in k(\gamma\circ\varphi(X))$, it is inferred that $y \in k(\gamma\circ\varphi (X))$ and the Gauss map $\gamma$ is birational onto its image. 
Assertion (3) follows. 
\end{proof} 

\begin{proof}[Proof of Theorem \ref{main2}]
Assume that conditions (a) and (b) are satisfied for an automorphism $\sigma$. 
By Proposition \ref{functions} (b) and (c), there exists $y \in V_{\sigma, x} \setminus \langle 1, x \rangle$ such that $k(X)=k(x, y)$. 
Conditions (a), (b), (c) and (d) in Theorem \ref{main} are satisfied.   
\end{proof}

\begin{remark}
\begin{itemize}
\item[(a)] The automorphism $\sigma$ is non-classical with respect to the plane model given by conditions (a), (b), (c) and (d) in Theorem \ref{main}.
\item[(b)] The same assertion as Theorem \ref{main} without condition (3) holds, if $p=2$.
\item[(c)] Assume that $p > 0$ and $N \ge 3$. 
By Proposition \ref{functions} and the same method as in the proof of Theorem \ref{main}, it can be proved that for any automorphism $\sigma \in {\rm Aut}(X) \setminus \{1\}$ of prime order, there exists a birational embedding $\varphi: X \rightarrow \mathbb{P}^N$ with conditions (1) and (2) as in Theorem \ref{main}. 
\end{itemize}
\end{remark}

Finally, this paper raises the following: 

\begin{question}
Assume that a smooth projective curve $C \subset \mathbb{P}^N$ not in any plane is tangentially degenerate and the Gauss map of $C$ is birational onto its image. 
Then is it true that $C$ is rational? 
\end{question} 

\begin{center} {\bf Acknowledgements} \end{center} 
The author is grateful to Professor Seiji Nishioka for helpful conversations, by which the author was able to improve Proposition \ref{functions}. 
The author thanks Professor Hajime Kaji for telling Question \ref{kaji's question} and useful comments.

\end{document}